\documentclass{amsart}
\usepackage{amssymb}
\usepackage{url}
\usepackage{mathrsfs}
\usepackage{booktabs}
\newtheorem{thm}{Theorem}
\newtheorem{cor}{Corollary}
\newtheorem{lem}{Lemma}

\theoremstyle{definition}

\theoremstyle{remark}

\begin{document}

\title[Perfect Numbers and Fibonacci Primes (II)]
{Perfect Numbers and Fibonacci Primes (II)}

\author{Tianxin Cai}
\address{Department of Mathematics, Zhejiang University, Hangzhou, 310027, People's Republic of China}
\email{txcai@zju.edu.cn}

\author{ LIUQUAN WANG}

\address{Department of Mathematics, National University of Singapore, Singapore, 119076, Singapore}
\email{mathlqwang@163.com}

\author{Yong Zhang}
\address{Department of Mathematics, Zhejiang University, Hangzhou, 310027, People's Republic of China}
\address{College of Mathematics and Computing Science, Changsha University of Science and Technology, Changsha, 410114, People's Republic of China}
\email{zhangyongzju@163.com}

\date{June 22, 2014}

\thanks{Project supported by the National Natural Science Foundation of China 11351002.}


\dedicatory{}


\begin{abstract}
In this paper, we study the diophantine equation
${{\sigma }_{2}}(n)-{{n}^{2}}=An+B$. We prove that except for finitely many computable solutions, all the solutions to this equation with $(A,B)=({{L}_{2m}},F_{2m}^{2}-1)$ are $n={{F}_{2k+1}}{{F}_{2k+2m+1}}$, where both ${{F}_{2k+1}}$  and ${{F}_{2k+2m+1}}$ are Fibonacci primes. Meanwhile, we show that the twin primes conjecture holds if and only if the equation ${{\sigma }_{2}}(n)-{{n}^{2}}=2n+5$ has infinitely many solutions.
\end{abstract}

\maketitle

\section{Introduction and Main Results}
 Since ancient Greek, a perfect number means a positive integers which  equals to the sum of its proper divisors. Equivalently, a positive integer $n$ is a perfect number if and only if
\[\sum\limits_{\begin{smallmatrix}
 d|n \\
 d<n
\end{smallmatrix}}{d}=n.\]
Let ${{\sigma }_{k}}(n)=\sum\limits_{d|n}{{{d}^{k}}}$, and usually ${{\sigma }_{1}}(n)$ is briefly written as $\sigma (n)$. The perfect numbers are exactly those integers $n$ which satisfy $\sigma (n)=2n.$ Euler proved that all even perfect numbers are $n={{2}^{p-1}}({{2}^{p}}-1)$, where both $p$ and ${{2}^{p}}-1$ are primes.

Some variants and generalizations of  perfect numbers have been studied by many mathematicians, including great Fermat, Descartes, Mersenne and Euler,  see \cite{Cai,Dickson,Guy}.

In spring 2013, the first author  defined the so-called $F$-perfect number $n$ which satisfies
\begin{equation}\label{Caieq}
\sum\limits_{\begin{smallmatrix}
 d|n \\
 d<n
\end{smallmatrix}}{{{d}^{2}}}=3n,
\end{equation}
and he called the original one $M$-perfect numbers (In China-Japan Number Theory Conference hold in Fukuoka, 2013, Japanese number theoriest
Kohji Matsumoto suggested to name them female perfect numbers and male perfect numbers).
  Cai, Chen and Zhang \cite{Cai} proved
\begin{thm}\label{Caithm}
All the solutions of $(\ref{Caieq})$ are $n={{F}_{2k-1}}{{F}_{2k+1}}$, where both ${{F}_{2k-1}}$  and ${{F}_{2k+1}}$ are Fibonacci primes.
\end{thm}

There are 5 known $F$-perfect numbers. They also showed that for any pair of positive integers  $(a,b) \ne (2,3)$ and $a \ge 2$, the equation
\[\sum\limits_{\begin{smallmatrix}
 d|n \\
 d<n
\end{smallmatrix}}{{{d}^{a}}}=bn\]
has only finitely many solutions.

In this paper, we study
\begin{equation}\label{general}
{{\sigma }_{2}}(n)-{{n}^{2}}=An+B,
\end{equation}
where $A$ and $B$ are given intergers.

The main theorem we obtain is
\begin{thm}\label{ABthm}
$\mathrm{(i)}$  If $A=0$ and $B=1$, then all the solutions of $(\ref{general})$ are $n=p$, where $p$  is prime. \\
$\mathrm{(ii)}$  If $A=1$ and $B=1$, then  all the solutions of  $(\ref{general})$ are $n={{p}^{2}}$, where $p$ is  prime. \\
$\mathrm{(iii)}$ If $(A,B)\ne (0,1)$ or $(1,1)$, then except for finitely many computable solutions in the range $n\le {{(|A|+|B|)}^{3}}$, all the solutions of $(\ref{general})$  are $n=pq$, where $p<q$ are primes which satisfy the following equation
\begin{equation}\label{pqeq}
{{p}^{2}}+{{q}^{2}}+(1-B)=Apq.
\end{equation}
\end{thm}
For  special pairs $(A,B)$,  (\ref{pqeq}) are solvable. Let $m$ be any positive integer, we have
\begin{thm}\label{F1}
Except for finitely many computable solutions in the range  $n\le {{(L_{2m}+F_{2m}^2-1)}^{3}}$, all the solutions of
\begin{equation}\label{Feq}
{{\sigma }_{2}}(n)-{{n}^{2}}={{L}_{2m}}n-(F_{2m}^{2}-1)
\end{equation}
are\\
$\mathrm{(i)}$  $n={{F}_{2k+1}}{{F}_{2k+2m+1}}(k\ge 0)$, both ${{F}_{2k+1}}$ and ${{F}_{2k+2m+1}}$ are Fibonacci primes;\\
$\mathrm{(ii)}$ $n={{F}_{2k+1}}{{F}_{2m-2k-1}}(0\le k<m,k\ne \frac{m-1}{2})$, both ${{F}_{2k+1}}$ and ${{F}_{2m-2k-1}}$ are Fibonacci primes.
\end{thm}
We call those solutions of (\ref{Feq}) which are not in the form of (i) or (ii) exceptional solutions.
For $1 \le m \le 5$, by  using Mathematica, we find that there are no  exceptional solutions.  In particular, if $m=1$, (\ref{Feq}) becomes (\ref{Caieq}), and Theorem \ref{F1} becomes Theorem \ref{Caithm}.  For $1 \le m \le 3$, by the list of known Fibonacci primes \cite{Fprime},  we could find  those known solutions.

\begin{table}[h]
\centering
\begin{tabular}{|c|c|c| }
\hline
  $m$ &  ${{L}_{2m}}n-(F_{2m}^{2}-1)$ &   $n$
\\ \hline
  1 &  $3n$  & ${{F}_{3}}{{F}_{5}},{{F}_{5}}{{F}_{7}},{{F}_{11}}{{F}_{13}},{{F}_{431}}{{F}_{433}},{{F}_{569}}{{F}_{571}}$    \\
\hline
  2 & $ 7n-8 $    & ${{F}_{3}}{{F}_{7}}, {{F}_{7}}{{F}_{11}}, {{F}_{13}}{{F}_{17}}, {{F}_{43}}{{F}_{47}}$                                            \\
\hline
3  &  $18n-63 $ &  ${{F}_{5}}{{F}_{11}}, {{F}_{7}}{{F}_{13}}, {{F}_{11}}{{F}_{17}},$ \\
&     & ${{F}_{17}}{{F}_{23}}, {{F}_{23}}{{F}_{29}},  {{F}_{131}}{{F}_{137}}$ \\
\hline
\end{tabular}
\caption{Solutions of (\ref{Feq}) for $1 \le m \le 3$.}
\label{L1solution}
\end{table}
The two solutions ${{F}_{431}}{{F}_{433}},{{F}_{569}}{{F}_{571}}$ in the above table has 180  and 238 digits, respectively.

The exceptional solutions may exist for larger $m$. For example, if $m=6$,  (\ref{Feq}) becomes
\[{{\sigma }_{2}}(n)-{{n}^{2}}=322n-20735.\]
There is exactly one exceptional solution $n=1755={{3}^{3}}\cdot 5\cdot 13.$ For other $m$, the exceptional solution may exist as well, but it would be  difficult  to find those  $m$ and $n$ such that  (\ref{Feq}) has no exceptional solutions.

Meanwhile, we have
\begin{thm}\label{L1}
Except for finitely many computable solutions in the range $n\le$ \\
$ {{(L_{2m}^{2}+{{L}_{2m}}-3)}^{3}}$, all the solutions of
\begin{equation}\label{L1eq}
{{\sigma }_{2}}(n)-{{n}^{2}}={{L}_{2m}}n+(L_{2m}^{2}-3)
\end{equation}
are $n={{L}_{2k-1}}{{L}_{2k+2m-1}}$, where both ${{L}_{2k-1}}$ and ${{L}_{2k+2m-1}}$ are Lucas primes.
\end{thm}

\begin{thm}\label{L2}
Except for finitely many computable solutions in the range $n\le$ \\
$ (L_{2m}+L_{2m}^{2}-5)^{3}$, all the solutions of
\begin{equation}\label{L2eq}
{{\sigma }_{2}}(n)-{{n}^{2}}={{L}_{2m}}n-(L_{2m}^{2}-5)
\end{equation}
are\\
$\mathrm{(i)}$  $n={{L}_{2k}}{{L}_{2k+2m}}(k\ge 0)$,  both ${{L}_{2k}}$ and ${{L}_{2k+2m}}$ are  Lucas primes.\\
$\mathrm{(ii)}$  $n={{L}_{2k}}{{L}_{2m-2k}}(0\le k\le m,k\ne \frac{m}{2})$, both ${{L}_{2k}}$ and ${{L}_{2m-2k}}$ are Lucas primes.
\end{thm}

Similarly, we  call those  solutions of (\ref{L1eq}) or (\ref{L2eq})  which are not the product of two Lucas primes exceptional solutions. From this two theorems and by Mathematica, we figure out that for (\ref{L1eq}) and (\ref{L2eq}), there are no exceptional solutions for $1\le m\le 5$.  The following table gives the known solutions of (\ref{L1eq}) with $1 \le m \le 5$.

\begin{table}[h]
\centering
\begin{tabular}{|c|c|c| }
\hline
  $m$ &  $L_{2m}n+(L_{2m}^{2}-3)$ &   $n$
\\ \hline
  1 & $3n+6$    & $L_{5}L_{7}, L_{11}L_{13}, L_{17}L_{19}$    \\
\hline
  2 & $7n+46$   & ${{L}_{7}}{{L}_{11}},{{L}_{13}}{{L}_{17}}, {{L}_{37}}{{L}_{41}}, {{L}_{613}}{{L}_{617}}$                                            \\
\hline
3  &  {$18n+321$}  &  ${{L}_{5}}{{L}_{11}}, {{L}_{7}}{{L}_{13}}, {{L}_{11}}{{L}_{17}}, {{L}_{13}}{{L}_{19}}, {{L}_{31}}{{L}_{37}},$  \\
  &  & $ {{L}_{41}}{{L}_{47}},   {{L}_{47}}{{L}_{53}}, {{L}_{4787}}{{L}_{4793}} $ \\
\hline
\end{tabular}
\caption{Solutions of (\ref{L1eq}) for $1 \le m \le  3$.}
\label{L1solution}
\end{table}
The two solutions $n = L_{613} L_{617}$ and $L_{4787} L_{4793}$ in above table have 258 and 2003 digits, respectively.

Finally, we have
\begin{thm}\label{twin}
Let $A$ and $k$ be positive integers, consider the equation
\begin{equation}\label{twineq}
{{\sigma }_{2}}(n)-{{n}^{2}}=An+({{k}^{2}}+1).
\end{equation}
$\mathrm{(i)}$  If $A\ne 2$ or $k$ is odd,  then $(\ref{twineq})$ has only finitely many solutions; \\
$\mathrm{(ii)}$  If $A=2$ and $k$ is even, then except for finitely many computable solutions in the range $n<{{(|A|+{{k}^{2}}+1)}^{3}}$,  all the solutions of $(\ref{twineq})$ are $n=p(p+k)$, where both $p$ and $p+k$ are primes.
\end{thm}
\begin{cor}
For any even integer $k$, there are infinitely many prime pairs $(p, p + k)$ (Polignac's conjecture)  if and only if the equation
\[ \sigma_{2}(n)-n^2=2n+(k^2+1)\]
has infinitely many solutions.
\end{cor}

In particular,  twin primes conjecture holds if and only if the equation
\[{{\sigma }_{2}}(n)-{{n}^{2}}=2n+5\]
has infinitely many solutions.


\section{Proofs of the Theorems}

In order to prove the theorems, we need the following two Lemmas.
\begin{lem}[cf. \cite{identity}]\label{FLid}
Let $m$, $n$ be  nonnegative integers and $n\ge m$, then \\
$\mathrm{(i)}$    $5F_{n}^{2}+4{{(-1)}^{n}}=L_{n}^{2}.$ \\
$\mathrm{(ii)}$   $L_{m}L_{n}+5F_{m}F_{n}=2L_{m+n}.$ \\
$\mathrm{(iii)}$  $F_{n}L_{m}=F_{n+m}+{{(-1)}^{m}}F_{n-m}.$ \\
$\mathrm{(iv)}$  $L_{n}F_{m}=F_{n+m}-{{(-1)}^{m}}F_{n-m}.$ \\
$\mathrm{(v)}$  $5{{F}_{m}}{{F}_{n}}={{L}_{m+n}}-{{(-1)}^{m}}{{L}_{n-m}}.$
\end{lem}

\begin{lem}[cf. Example 4, page 51 of \cite{Titu}]\label{Pell}
Suppose that $a$ and $b$ are  positive integers satisfying ${{a}^{2}}-5{{b}^{2}}=-4$, then there exists an  integer $k \ge 0$ such that $a={{L}_{2k+1}}$ and $b={{F}_{2k+1}}$.
\end{lem}

\begin{proof}[\textbf{Proof of Theorem \ref{ABthm}}]
(i) is obvious, since ${{\sigma }_{2}}(n)={{n}^{2}}+1$ if and only if $n$ is a prime.
To prove (ii) and (iii), we  might assume $(A,B)\ne (0,1)$.

 Suppose $n>{{(|A|+|B|)}^{3}}$ is a solution of (\ref{general}).

If $n=abc$ where $1<a<b<c$ are positive integers. By the arithmetic-geometric mean inequality, we have
\[{{\sigma }_{2}}(n)-{{n}^{2}}\ge {{a}^{2}}{{b}^{2}}+{{b}^{2}}{{c}^{2}}+{{c}^{2}}{{a}^{2}}\ge 3{{({{a}^{4}}{{b}^{4}}{{c}^{4}})}^{1/3}}=3{{n}^{\frac{4}{3}}}.\]
Thus
\[{{\sigma }_{2}}(n)-{{n}^{2}} \ge 3{{n}^{4/3}}>n\cdot {{n}^{1/3}}>n(|A|+|B|)\ge An+B.\]
A contradiction! Hence  $n$ couldn't be a  product of three distinct positive integers which are greater than 1.

Let $\omega (n)$ denotes the number of distinct prime factors of $n$. If $\omega (n)\ge 3$, then obviously we can write $n=abc$ where $1<a<b<c$. A contradiction! From now on, we suppose $\omega (n)\le 2$.

\textbf{Case 1:} If $\omega (n)=1$, write $n={{p}^{\alpha }}$ where $p$ is a prime. Since $n=p\cdot {{p}^{2}}\cdot {{p}^{\alpha -3}}$, we must have $\alpha \le 5$.

If $\alpha =1$, (\ref{general}) becomes $Ap+B=1$, which  has at most one solution $p=(1-B)/A.$

If $\alpha =2$,  from (\ref{general}) we deduce that ${{p}^{2}}(1-A)=B-1$. If $A=1$ and $B=1$, there are infinitely many solutions $n={{p}^{2}}$, where $p$ could be  any prime.  If $A=1$ and $B\ne 1$, there are no solutions. If $A\ne 1$, then  (\ref{general}) has at most one solution $n={{p}^{2}}=(B-1)/(1-A).$

If $3\le \alpha \le 5$, from (\ref{general}) we deduce that  $p(p+{{p}^{3}}+\cdots +{{p}^{2\alpha -3}}-A{{p}^{\alpha -1}})=B-1$. Note that
\[p+{{p}^{3}}+\cdots +{{p}^{2\alpha -3}}-A{{p}^{\alpha -1}}\equiv p \, (\bmod \, {{p}^{2}}),\]
 hence $p+{{p}^{3}}+\cdots +{{p}^{2\alpha -3}}-A{{p}^{\alpha -1}}\ne 0$, and  $p^{2}|B-1$. If $A=0,B=2$ or $3$, it is easy to see that (\ref{general}) has no solutions. For other situations, we know  the possible solution is
$n={{p}^{\alpha }}$ with $3\le \alpha \le 5$ and ${{p}^{2}}|B-1$,  but then  $n=p^{\alpha} \le p^{5} \le {{(1+|B|)}^{5/2}}\le {{(|A|+|B|)}^{3}}$. A contradiction!

\textbf{Case 2:} If $\omega (n)=2$, we write $n={{p}^{\alpha }}{{q}^{\beta }}$. If $\alpha \ge 3$, then $n=p\cdot {{p}^{\alpha -1}}\cdot {{q}^{\beta }}$.  A contradiction! Therefore $\alpha \le 2$, with the same reason,  $\beta \le 2$.

If $(\alpha ,\beta )=(2,2)$, we can write $n=p\cdot q\cdot pq$.  A contradiction!

If $(\alpha ,\beta )=(1,2)$, then
\[{{\sigma }_{2}}(n)-{{n}^{2}}=1+{{p}^{2}}+{{p}^{2}}{{q}^{2}}+{{q}^{2}}+{{q}^{4}}=Ap{{q}^{2}}+B\le (|A|+|B|)p{{q}^{2}},\]
this implies $p<|A|+|B|,{{q}^{2}}<(|A|+|B|)p$,  hence
\[n=p{{q}^{2}}<(|A|+|B|)^{2}p<{{(|A|+|B|)}^{3}}.\]
A contradiction!

Similarly if $(\alpha ,\beta )=(2,1)$, we have  $n<{{(|A|+|B|)}^{3}}.$

Finally, if $(\alpha ,\beta )=(1,1)$, (\ref{pqeq}) follows immediately from (\ref{general}), with this we prove  (iii).

From the above argument, for $(A,B)=(1,1)$, it's easy to see that (\ref{pqeq}) has no solutions. Hence if $n > 8$, all the solutions of ${{\sigma }_{2}}(n)-{{n}^{2}}=n+1$ are $n=p^2$ with $p$ a prime. If $n \le 8$, the only solution is $n=4$,  with this we prove (ii).
\end{proof}

To prove Theorems \ref{F1}-\ref{L2}, we first observe that they are actually different assignment of $(A,B)$ in (\ref{general}). By Theorem \ref{ABthm} we know that except for those computable solutions $n \le {{(|A|+|B|)}^{3}}$, all the solutions are of the form $n=pq$, where $p$, $q$ are distinct primes satisfying (\ref{pqeq}). Note that (\ref{pqeq}) is equivalent to
\begin{equation}\label{cond}
{{(2p-Aq)}^{2}}-({{A}^{2}}-4){{q}^{2}}=4(B-1).
\end{equation}

\begin{proof}[\textbf{Proof of Theorem \ref{F1}}]
Taking $A={{L}_{2m}}$, $B=-F_{2m}^{2}+1$ in (\ref{cond}), then
\[{{(2p-{{L}_{2m}}q)}^{2}}-(L_{2m}^{2}-4){{q}^{2}}=-4F_{2m}^{2}.\]
By (i) of Lemma \ref{FLid}, the above equation becomes ${{(2p-{{L}_{2m}}q)}^{2}}-5F_{2m}^{2}{{q}^{2}}=-4F_{2m}^{2},$ this implies ${{F}_{2m}}|2p-{{L}_{2m}}q$. Write $2p-{{L}_{2m}}q=u{{F}_{2m}}$, where $u$ is an integer, then we deduce that ${{u}^{2}}-5{{q}^{2}}=-4$. By Lemma \ref{Pell} we have $(u,q)=(\pm {{L}_{2k+1}},{{F}_{2k+1}})$ for some nonnegative integer $k.$

\textbf{Case 1:}  $(u,q)=({{L}_{2k+1}},{{F}_{2k+1}})$,  then $p=\frac{1}{2}({{L}_{2m}}{{F}_{2k+1}}+{{L}_{2k+1}}{{F}_{2m}}).$
By (iii) and (iv) of Lemma \ref{FLid} we have $p={{F}_{2k+1+2m}}.$ Hence $n={{F}_{2k+1}}{{F}_{2k+2m+1}}$, where both $F_{2k+1}$ and $F_{2k+2m+1}$ are  primes.

\textbf{Case 2:} $(u,q)=(-{{L}_{2k+1}},{{F}_{2k+1}})$, then $p=\frac{1}{2}({{L}_{2m}}{{F}_{2k+1}}-{{L}_{2k+1}}{{F}_{2m}}).$

If $2k+1>2m$,  by (iii) and (iv) of Lemma \ref{FLid} we have  $p={{F}_{2k+1-2m}}$. Hence $n={{F}_{2k+1}}{{F}_{2k+1-2m}}.$

If $2k+1<2m$,  by (iii) and (iv) of Lemma \ref{FLid} we have $p={{F}_{2m-2k-1}}$. Hence $n={{F}_{2k+1}}{{F}_{2m-2k-1}}.$
\end{proof}

\begin{proof}[\textbf{Proof of  Theorem \ref{L1}}]
Taking $A={{L}_{2m}}$, $B=L_{2m}^{2}-3$ in (\ref{cond}), then
\begin{equation}\label{mid}
{{(2p-{{L}_{2m}}q)}^{2}}-(L_{2m}^{2}-4){{q}^{2}}=4(L_{2m}^{2}-4).
\end{equation}
By (i) of Lemma \ref{FLid}, we have $L_{2m}^{2}-4=5F_{2m}^{2}$, hence it follows from (\ref{mid}) that  $2p-{{L}_{2m}}q=5{{F}_{2m}}u$ for some integer $u$.  Then (\ref{mid}) becomes ${{q}^{2}}-5{{u}^{2}}=-4.$ By Lemma \ref{Pell} we have  $(q,u)=({{L}_{2k+1}},\pm {{F}_{2k+1}})$ for some integer $k\ge 0$.

\textbf{Case 1: }$(q,u)=({{L}_{2k+1}},{{F}_{2k+1}})$,  by (ii) of Lemma \ref{FLid} we have
$p=\frac{1}{2}({{L}_{2m}}{{L}_{2k+1}}+5{{F}_{2m}}{{F}_{2k+1}})={{L}_{2m+2k+1}}.$  Hence,  $n={{L}_{2k+1}}{{L}_{2k+2m+1}}.$

\textbf{Case 2: }$(q,u)=({{L}_{2k+1}},-{{F}_{2k+1}})$, then
\[p=\frac{1}{2}({{L}_{2m}}{{L}_{2k+1}}-5{{F}_{2m}}{{F}_{2k+1}})=\frac{1}{2}({{L}_{2m}}{{L}_{2k+1}}+5{{F}_{2m}}{{F}_{2k+1}})-5{{F}_{2m}}{{F}_{2k+1}}.\]

If $2m>2k+1$,  by  (ii) and (v) of Lemma \ref{FLid} we have  $p=-{{L}_{2m-2k-1}}$, which is impossible!

If $2m<2k+1$,  by (ii) and (v) of  Lemma \ref{FLid} we have $p={{L}_{2k+1-2m}}$, hence $n={{L}_{2k+1}}{{L}_{2k+1-2m}},$ which is the same as in Case 1.
\end{proof}

\begin{proof}[\textbf{Proof of Theorem \ref{L2}}]
Taking  $A={{L}_{2m}}$ and $B=5-L_{2m}^{2}$ in  (\ref{cond}), then
\[{{(2p-{{L}_{2m}}q)}^{2}}-(L_{2m}^{2}-4){{q}^{2}}=4(4-L_{2m}^{2}).\]
By (ii) of Lemma \ref{FLid},  we have $L_{2m}^{2}-4=5F_{2m}^{2}$, thus we can write $2p-{{L}_{2m}}q=5{{F}_{2m}}u$ for some integer $u$. The above equation becomes ${{q}^{2}}-5{{u}^{2}}=4$. Solving this Pell-type equation we deduce that $(q,u)=({{L}_{2k}},\pm {{F}_{2k}})$ for some  integer $k \ge 0$.

\textbf{Case 1:} $(q,u)=({{L}_{2k}},{{F}_{2k}})$. By Lemma \ref{FLid} we have $p=\frac{1}{2}({{L}_{2m}}{{L}_{2k}}+5{{F}_{2m}}{{F}_{2k}})={{L}_{2m+2k}}$, hence   $n={{L}_{2k}}{{L}_{2k+2m}}$.

\textbf{Case 2: }$(q,u)=({{L}_{2k}},-{{F}_{2k}})$, then
\[p=\frac{1}{2}({{L}_{2m}}{{L}_{2k}}-5{{F}_{2m}}{{F}_{2k}})=\frac{1}{2}({{L}_{2m}}{{L}_{2k}}+5{{F}_{2m}}{{F}_{2k}})-5{{F}_{2m}}{{F}_{2k}},\]

If $m>k$, then by Lemma \ref{FLid} we have $p={{F}_{2m-2k}}$, and hence $n={{F}_{2m-2k}}{{F}_{2k}}$, note that $p,q$ are distinct primes, therefore $k\ne \frac{m}{2}.$

If $m\le k$, then by Lemma \ref{FLid} we have $p={{F}_{2k-2m}}$, and hence $n={{F}_{2k-2m}}{{F}_{2k}}$.
\end{proof}

\begin{proof}[\textbf{Proof of Theorem \ref{twin}}]
Without loss of generality we  assume $k\ge 0$. Let $n>{{(|A|+{{k}^{2}}+1)}^{3}}$ be a solution of (\ref{twineq}), it follows from (iii) of Theorem \ref{ABthm} that $n=pq$, where $p, q$ are distinct primes satisfying ${{p}^{2}}+{{q}^{2}}-{{k}^{2}}=Apq$. Let $p <q$,  note that $q|(p-k)(p+k)$.

If $p=k$, then we have $q=Ak$ and $n=A{{k}^{2}}<{{(|A|+{{k}^{2}}+1)}^{3}}$.  A contradiction!

If $p<k$, then $p+k<2k$.  Since $p<q$,  we must have $q|p+k$ and hence $n<2{{k}^{2}}<{{(|A|+{{k}^{2}}+1)}^{3}}$. A contradiction!

If $p>k$, then $q|p+k$, note that $2q>2p>p+k$, it must be  $q=p+k$. Thus  $A=({{p}^{2}}+{{q}^{2}}-{{k}^{2}})/pq=2$ and $n=p(p+k)$. For $p\ge 3$,   $k$ must be even.

Conversely, if $A=2$, $k$ is even and both $p$ and $p+k$ are primes,  then  it's easy to see that $n=p(p+k)$ is a solution of (\ref{twineq}).
\end{proof}

Finally, we raise the following

\textbf{Question.} Let $a,b$ and $c$ be given integers, when does the  equation
\[{{\sigma }_{2}}(n)=a{{n}^{2}}+bn+c\]
has infinitely many solutions and how to solve this equation?

 We have discussed the case $a=1$. For $a \ge 2$, it's still unknown.


\end{document}